\documentclass[11pt,a4,reqno]{amsart}
\usepackage{amsfonts,amsmath,amssymb}
\usepackage{mathabx}
\usepackage{mathrsfs}
\usepackage[latin1]{inputenc}
\usepackage[usenames,dvipsnames]{pstricks}
\newtheorem{theorem}{Theorem}[section]

\newtheorem{lemma}[theorem]{Lemma}

\newcommand{\R}{\mathbb{R}}

\newcommand{\N}{\mathbb{N}}

\newcommand{\s}{\mathbb{S}}

\newcommand{\norm}[1]{\|#1\|}
\newcommand{\abs}[1]{|#1|}
\newcommand{\set}[1]{\left\{#1\right\}}
\newcommand{\para}[1]{\left(#1\right)}

\newcommand{\seq}[1]{\left<#1\right>}

\newcommand{\To}{\longrightarrow}

\usepackage{graphicx}

\newcommand{\F}{\mathcal{F}}
\newcommand{\rr}{\varrho}
\newcommand{\E}{E^t_\rr}
\newcommand{\beas}{\begin{eqnarray*}}
\newcommand{\eeas}{\end{eqnarray*}}
\newcommand{\V}{\mathcal{V}_\delta}
\newcommand{\LL}{L^2(\R^n)}
\usepackage{times}
\allowdisplaybreaks
\begin{document}

\title[Stability estimate in the inverse scattering]{Stability estimate in the inverse scattering for a single quantum particle in an external short-range potential}
\author[M. Bellassoued]{Mourad Bellassoued}
\author[L.Robbiano]{Luc Robbiano}
\address{University of Tunis El Manar, National Engineering School of Tunis, ENIT-LAMSIN, B.P. 37, 1002 Tunis, Tunisia}
\email{mourad.bellassoued@enit.utm.tn}
\address{Laboratoire de Math\'ematiques, Universit\'e de Versailles Saint-Quentin en Yvelines,
78035 Versailles, France}
\email{luc.robbiano@uvsq.fr}
\date{\today}

\begin{abstract}
In this paper we consider the inverse scattering problem for the Schrödinger operator with short-range electric potential. We prove in dimension $n\geq 2$ that the knowledge of the scattering operator determines the electric potential and we establish Hölder-type stability in determining the short range electric potential.
\end{abstract}
\maketitle
\section{Introduction and main results}
This paper concerns inverse scattering problems for a large class of Hamiltonian with short-range electric potential. A single quantum particle in an external potential is described by the Hilbert space $\LL$ and the family of Schrödinger Hamiltonian
\begin{equation}\label{1.1}
H=-\frac{1}{2}\Delta+V(x),\quad x\in\R^n.
\end{equation}
We suppose that the electric potential $V\in\mathcal{C}^1(\R^n,\R)$, with the short-range condition
$$
\abs{ V(x)}\leq C\seq{x}^{-\delta},
$$
for some $\delta>1$, where $\seq{x}=(1+\abs{x}^2)^{1/2}$. Then we define
$$
\V=\set{V\in\mathcal{C}^1(\R^n),\quad \abs{V(x)}\leq C\seq{x}^{-\delta},\quad \delta>1\,}.
$$
Let $H_0=\frac{1}{2}\Delta$ be the free Hamiltonian. We consider two strongly continuous unitary groups: $e^{-itH_0}$ generate the free dynamic of the system and $e^{-itH}$ a perturbation of this free dynamic. The state $u\in\LL$ is said asymptotically free as $t\to\pm\infty$ if there exists $\psi_\pm\in \LL$ such that
\begin{equation}\label{1.3.1}
\lim_{t\to\pm\infty}\norm{e^{-itH}u-e^{-itH_0}\psi_\pm}=0.
\end{equation}
Here $\psi_+$ is the outgoing (resp. incoming) asymptotic of the state $u$. The condition \eqref{1.3.1} is equivalent to the following two conditions
$$
\lim_{t\to\pm\infty}\norm{e^{itH_0}e^{-itH}u-\psi_\pm}=0,\quad \lim_{t\to\pm\infty}\norm{e^{itH}e^{-itH_0}\psi_\pm-u}=0.
$$
The fundamental direct problems of scattering theory are: (a) to determine the set of asymptotically free states, i.e., the set of $u\in \LL$ such that
$$
\lim_{t\to\pm\infty}e^{itH_0}e^{-itH}u=\psi_\pm
$$
exist, (b) the condition of the scattering operator which maps the incoming $\psi_-$ into the corresponding outgoing one $\psi_+$. 
\medskip

Let $V$ be a short-range electric potential, by \cite{Hormander}, Theorem 14.4.6,  the wave operators, defined by
$$
W_\pm(H,H_0) u=\lim_{t\to\pm\infty} e^{itH}e^{-itH_0}u,\quad u\in\LL
$$
exist as strong limits, are isometric operators, they intertwine the free and full Hamiltonian $H$ and $H_0$
$$
W_\pm(H,H_0) H_0=HW_\pm(H,H_0).
$$
Their range is the projection of the space $\LL$ onto continuous spectrum. Moreover the wave operators $W_\pm(H_0,H)$ also exist and adjoints to $W_\pm(H,H_0)$. The scattering operator $S_V:\psi_-\mapsto\psi_+$ is defined as
$$
S_V=W_+(H_0,H)W_-(H,H_0)=W_+(H,H_0)^*W_-(H,H_0).
$$
It is well known that $S_V$ is a unitary operator on $\LL$. We call $\mathcal{S}$ as a mapping from $\V$ into the set of bounded operators $\mathcal{L}(\LL)$, $\mathcal{S}(V)=S_V$, the scattering map.
\smallskip

For $s>0$, introducing the space $L^1_s(\R^n)$ be the weighted $L^1$ space in $\R^n$ with norm
$$
\norm{u}_{L^1_s(\R^n)}=\norm{\seq{\,\cdot\,}^s u}_{L^1(\R^n)}.
$$
\smallskip

The following is the main result of this paper.
\begin{theorem}\label{T.1}
Let $M>0$, $\delta>1$ and $s\in (0,1)$. There exist constants $C>0$ and $\nu\in (0,1)$ such that the following stability estimate holds
\begin{equation}\label{1.7}
\norm{V_1-V_2}_{H^{-1}(\R^n)}\leq C\norm{S_{V_1}-S_{V_2}}_{\mathcal{L}(\LL)}^\nu
\end{equation}
for every $V_1,\,V_2\in\V$ such that $(V_1-V_2)\in \LL\cap L^1_{s}(\R^n)$ and 
\begin{equation}
\norm{V}_{\LL}+\norm{V}_{L^1_s(\R^n)}\leq M.
\end{equation}
 In particularly the scattering map 
$$
\mathcal{S}:\V \To\mathcal{L}(\LL),\quad V\longmapsto S_V,
$$ 
is locally injective.
\end{theorem}
\medskip

We describe now some previous results related with our  problem. Let $\s^{n-1}$ be the unit sphere in $\R^n$. Define the unitary operator 
$$
\mathscr{F}:\LL\To L^2(\R^+,L^2(\s^{n-1})),\quad \mathscr{F}(u)(\omega,\lambda)=2^{-1/2}\lambda^{n-2/4}\hat{u}(\sqrt{\lambda}\omega),
$$
where $L^2(\R^+,L^2(\s^{n-1}))$ denote the $L^2$-space of functions defined on $\R^+$ with value in $L^2(\s^{n-1})$. The spectral parameter $\lambda$ plays the role of the energy of a quantum particle. Then
$$
\mathscr{F}(S_Vu)(\lambda)=S_V(\lambda)\mathscr{F}(u)(\lambda).
$$ 
The unitary operator $S_V(\lambda): L^2(\s^{n-1})\to L^2(\s^{n-1})$ is called the scattering matrix at fixed energy $\lambda$ with respect the electric potential $V$.
\medskip

The problem of identifying coefficients appearing in Schrödinger equation was treated
very well and there are many works that are relevant to this topic. In the case of a compactly supported electric potential and in dimension $n\geq 3$ uniqueness for the fixed energy scattering problem was given in \cite{18,22,25}. In the earlier paper \cite{24} this was done for small potentials. It is well known that for compactly supported potentials, knowledge of the scattering amplitude (or the scattering matrix) at fixed energy $\lambda$ is equivalent to knowing the Dirichlet-to-Neumann map for the Schrödinger equation measured on the boundary of a large ball containing the support of the potential (see \cite{34} for an account). Then the uniqueness result of Sylvester and Uhlmann \cite{32} for the Dirichlet-to-Neumann map, based on special solution called complex geometrical optics solutions, implies uniqueness at a fixed energy for compactly supported potentials. Melrose \cite{17} proposed a related proof that uses the density of products of scattering solutions.
\medskip

The uniqueness result with fixed energy was extended by Novikov to the case of exponentially decaying potentials \cite{23}. Another proof applying arguments similar to the ones used for studying the Dirichlet-to-Neumann map was given in \cite{35}. The fixed energy result for compactly supported potentials in the two-dimensional case follows from the corresponding uniqueness result for the Dirichlet-to-Neumann map of Bukhgeim \cite{2}, and this result was recently extended to potentials decaying faster than any Gaussian in \cite{8}.
\medskip

We note that in the absence of exponential decay for the potentials, there are counterexamples to uniqueness for inverse scattering at fixed energy. In two dimensions Grinevich and Novikov \cite{7} give a counterexample involving $V$ in the Schwartz class, and in dimension three there are counterexamples with potentials decaying like $|x|^{-3/2}$ \cite{21,27}. However, if the potentials have regular behavior at infinity (outside a ball they are given by convergent asymptotic sums of homogeneous functions in the radial variable), one still has uniqueness even in the magnetic case by the results of Weder and Yafaev \cite{37,38} (see also Joshi and Sá Barreto \cite{12,14}).
\medskip

In the case of two-body Schrödinger Hamiltonians $H$ with $V$ short range, such a problem has been studied in \cite{4} with high-frequency asymptotic methods. For short or long-range potentials, Enss and Weder \cite{5} have used a geometrical method. They show that the potential is uniquely recovred by the high-velocity limit of the scattering operator. This method can be used to study Hamiltonians with electric and magnetic potentials on $L^2(\R^n)$, the Dirac equation, \cite{7} and the $N$-body case \cite{5}. In \cite{11}, Nicoleau used a stationary method to study Hamiltonians with smooth electric and magnetic potentials have to be $\mathcal{C}^\infty$ functions with stronger decay assumption on higher derivatives, based on the construction of suitable modified wave operators. This approach gives the complete asymptotic expansion of the Scattering operator at high energies. In \cite{12} the author sees that the problem with obstacles can be treated in the same way by determining a class of test functions which have negligible interaction with the obstacle.
\medskip

All the mentioned papers are concerned only with uniqueness or reconstruction formula of the coefficients. Inspired by the work of Enss and Weder \cite{5} and following the same strategy as in \cite{5}, we prove in this paper stability estimates in the recovery of the unknown coefficient $V$ via the scattering map.
\medskip

The paper is organized as follows. In Section 2 we examine the scattering  problem associated with \eqref{1.1}, by using the geometric time-dependent method developed by Enss and Weder. In Section 3, we prove some intermediate estimate of the $X$-ray transform of the potential $V$. In Section 4, we
estimate the $X$-ray transform and the Fourier transform of the potential, in terms of the scattering map and we proof Theorem \ref{T.1}. 
\section{Scattering map}
\setcounter{equation}{0}
Here we recall some basic definitions of the scattering theory used throughout the paper. The Fourier transform on functions in $\R^n$ is defined by
$$
\hat{f}(\xi):=\F(f)(\xi)=\frac{1}{(2\pi)^{n/2}}\int_{\R^n} e^{-ix\cdot\xi}f(x)dx,
$$
and the inverse Fourier transform is
$$
f(x)=\F^{-1}(\hat{f})(x)=\frac{1}{(2\pi)^{n/2}}\int_{\R^n} e^{ix\cdot\xi}\hat{f}(\xi)d\xi.
$$
For $s\geq 0$, letting  $H^s(\R^n)$ stand for the standard Sobolev space of those measurable functions $f$ whose Fourier transform $\hat{f}$ satisfies
$$
\norm{f}_{H^s(\R^n)}=\left(\int_{\R^n}\seq{\xi}^{2s}\abs{\hat{f}(\xi)}^2d\xi\right)^{1/2}<\infty,\quad \seq{\,\cdot\,}=(1+\abs{\cdot}^2)^{1/2}.
$$
For $\delta>0$, introducing the Hilbert space $L^2_\delta(\R^n)$ be the weighted $\LL$ space in $\R^n$ with norm
$$
\norm{u}_{L^2_\delta(\R^n)}=\norm{\seq{\,\cdot\,}^\delta u}_{\LL}.
$$
We see that the Fourier transform $\F$ is a unitary transformation from $H^s(\R^n)$ onto $L_s^2(\R^n)$, that is 
\begin{equation}\label{2.3.1}
\norm{u}_{L^2_\delta(\R^n)}=\norm{\F(u)}_{H^\delta(\R^n)},\quad \forall u\in\mathcal{S}(\R^n).
\end{equation}
Let $e^{-itH_0}$ be the Schrödinger propagator, in term of the Fourier transform, this is given by
\begin{equation}\label{2.4}
e^{-itH_0}u=\F^{-1}(e^{-it\frac{\abs{\xi}^2}{2}}\F(u))(x)=\frac{1}{(2\pi)^{n/2}}\int_{\R^n} e^{ix\cdot\xi}e^{-it\frac{\abs{\xi}^2}{2}}\hat{u}(\xi)d\xi.
\end{equation}
We also record the following properties of the wave operators $W_\pm$
\begin{equation}\label{2.5}
W_\pm^*W_\pm=I,\quad e^{-itH}W_\pm=W_\pm e^{-itH_0}.
\end{equation}
By Duhamel's formula, we have
\begin{equation}\label{2.6}
W_\pm=I+i\int_0^{\pm\infty} e^{itH}V e^{-itH_0}\, dt.
\end{equation}
The proof of \eqref{2.6} proceeds  by differentiation and subsequent integration: For $u\in \mathscr{D}(H_0)=\mathscr{D}(H)$ one has the product rule 
\begin{eqnarray*}
\frac{d}{dt}\left(e^{itH}e^{-itH_0}u\right)&=&e^{itH}iHe^{-itH_0}u-e^{itH}iH_0e^{-itH_0}u\cr
&=&ie^{itH}Ve^{-itH_0}u.
\end{eqnarray*}
This is now integrated to yield
$$
e^{itH}e^{-itH_0}u-u=i\int_0^t e^{isH}Ve^{-isH_0}uds,
$$
from which \eqref{2.6} follows after taking the limit $t\to\infty$.\\
Then from \eqref{2.6}, we find out that
\begin{equation}\label{2.7}
(W_+-W_-)u=i\int_{-\infty}^{\infty}  e^{itH}V e^{-itH_0} u\, dt,
\end{equation}
for any state $u\in\LL$ for which the integral is well defined. We have a similar formula for $W_\pm^*$
$$
W^*_\pm=I+i\int^0_{\pm\infty} e^{itH_0}V e^{-itH}\, dt.
$$
It follows from the definition of the scattering operators that
$$
S_V-I=(W_+-W_-)^*W_-.
$$
Then by Duhamel's formula and the interwining relation \eqref{2.5}, we have the following identity giving a relation between the scattering operator $S_V$ and the potential $V$
\begin{equation}\label{2.10}
i(S_V-I)u=\int_{-\infty}^{+\infty} e^{itH_0}V W_- e^{-itH_0}u\, dt,\quad u\in\LL.
\end{equation}
We need some elementary facts about pseudo-differential operators defined by the equality
$$
a(D)u(x)=\frac{1}{(2\pi)^{n/2}}\int_{\R^n} e^{ix\cdot\xi} a(\xi)\hat{u}(\xi)d\xi,\quad \forall\,u\in\mathcal{S}(\R^n),
$$
where the symbol $a\in\mathcal{C}_0^\infty(\R^n)$. It is then known that for any $a\in\mathcal{C}_0^\infty(\R^n)$, $a(D)$ is bounded operator on $\LL$.\\
For $\rr\in\R^n$, we define the conjugate pseudo-differential operator $a_\rr(D)$ by
\begin{equation}\label{2.12}
a_\rr(D)=e^{-ix\cdot\rr}a(D)e^{ix\cdot\rr}:=a(D+\rr).
\end{equation}
The symbol of the operator $a_\rr(D)$ is given by $a_\rr(\xi)=a(\xi+\rr)$. Indeed, using the fact that $\F(e^{ix\cdot\rr}u)(\xi)=\hat{u}(\xi-\rr)$, we get
$$
a_\rr(D)u(x)=\frac{1}{(2\pi)^{n/2}}\int_{\R^n} e^{ix\cdot(\xi-\rr)}a(\xi)\hat{u}(\xi-\rr)d\xi,\quad \forall\,u\in\mathcal{S}(\R^n).
$$
We define the linear unitary operator $\E$ from $\LL$ into itself by the integral representation 
$$
\E u(x)=e^{-it\rr\cdot D} u(x)=\frac{1}{(2\pi)^{n/2}}\int_{\R^n} e^{ix\cdot\xi} e^{-it\rr\cdot\xi}\hat{u}(\xi)d\xi=u(x-t\rr),\quad \forall\,u\in\mathcal{S}(\R^n).
$$
Hence, we obtain the following identity 
\begin{equation}\label{2.15}
E_{-\rr}^t\, w\, \E u(x)=w(x+t\rr)u(x), \quad \forall\,w,\,u\in\LL.
\end{equation}
By a simple calculation, it is easy to see that
\begin{equation}\label{2.16}
e^{-ix\cdot\rr}e^{-itH_0}e^{ix\cdot\rr}=e^{-it\abs{\rr}^2}\E e^{-itH_0},\quad\textrm{in}\,\LL.
\end{equation}
Let us recall the following result proved in Reed and Simon \cite{Reed-Simon}, XI, page 39. The key of the proof is the application of the stationary phase method.
\begin{lemma}\label{L2.1}
Let $g\in\mathcal{S}(\R^n)$ be a function such that $\hat{g}$ has a compact support. Let $\mathcal{O}$ be an open set containing the compact $\textrm{Supp}(\hat{g})$. Let 
\begin{equation}\label{2.17}
g_t(x)=\frac{1}{(2\pi)^{n/2}}\int_{\R^n} e^{ix\cdot\xi}e^{-i\frac{t}{2}\abs{\xi}^2}\hat{g}(\xi)d\xi.
\end{equation}
Then, for any $m\in\N$, there exists $C>0$ depending on $m$, $g$ and $\textrm{Supp}(\hat{g})$ so that
$$
\abs{g_t(x)}\leq C(1+\abs{x}+\abs{t})^{-m},
$$
for all $x,t$ with $xt^{-1}\notin \mathcal{O}$.
\end{lemma}
In the sequel, for $t\in\R^*$ and $\rr\in\R^n$, we denote by $A_1$ and $A_2$ the following sets
\begin{equation}\label{2.19}
A_1=\set{\abs{x-t\rr}>\frac{1}{2}\abs{t\rr}},\quad A_2=\set{\abs{x}<\frac{1}{4}\abs{t\rr}}.
\end{equation}
For a measurable set $A\subset\R^n$, we denote by $\kappa_A$ the characteristic function of $A$.
\begin{lemma}\label{L2.2}
Let $\rr\in\R^n$ such that $\abs{\rr}>4$, $t\in\R^*$, and let consider the two measurable sets $A_1$ and $A_2$ given by \eqref{2.19}. Then for any $a\in\mathcal{C}_0^\infty(B(0,1))$, and all $k\in\N$ there exists $C$ such that
\begin{equation}\label{2.20}
\norm{\kappa_{A_1}e^{-itH_0}a_{-\rr}(D)\kappa_{A_2} u}_{\LL}\leq C\seq{t\rr}^{-k}\norm{u}_{\LL},
\end{equation}
for any $u\in\LL$. Here $C$ depends only on $k$ and $n$ but not depends on $\rr$.
\end{lemma}
\begin{proof}
Let  $a\in\mathcal{C}_0^\infty(B(0,1))$, $A_1$ and $A_2$ given by \eqref{2.19}. For $u\in\mathcal{S}(\R^n)$, using \eqref{2.4} and \eqref{2.17}, we easily see that
\begin{eqnarray*}
\kappa_{A_1}e^{-itH_0}a_{-\rr}(D)\kappa_{A_2} u(x) &=&\frac{1}{(2\pi)^{n/2}}\int_{\R^n}e^{ix\cdot\xi} e^{-i\frac{t}{2}\abs{\xi}^2}\kappa_{A_1}(x)a(\xi-\rr)\F(\kappa_{A_2}u)(\xi)\,d\xi\cr
&=&\frac{1}{(2\pi)^{n}}\int_{\R^n}e^{ix\cdot\xi} e^{-i\frac{t}{2}\abs{\xi}^2}\kappa_{A_1}(x)a(\xi-\rr)\int_{\R^n}e^{-iy\cdot\xi}\kappa_{A_2}(y)u(y)dy\,d\xi\cr
&=& \frac{1}{(2\pi)^{n}}\int_{\R^n}\kappa_{A_1}(x)\kappa_{A_2}(y)\tilde{a}_t^{\rr}(x-y)u(y)dy\,,
\end{eqnarray*}
where the kernel $\tilde{a}_t^{\rr}$ is given by
$$
\tilde{a}_t^{\rr}(z)=\int_{\R^n}e^{iz\cdot\xi} e^{-i\frac{t}{2}\abs{\xi}^2}a(\xi-\rr)d\xi.
$$
Therefore, we have
\begin{eqnarray}\label{2.23}
\norm{\kappa_{A_1}e^{-itH_0}a_{-\rr}(D)\kappa_{A_2} u}^2_{\LL} &=& \frac{1}{(2\pi)^{2n}}\int_{\R^n}\kappa_{A_1}(x)\abs{\int_{\R^n}\kappa_{A_2}(y)\tilde{a}_t^{\rr}(x-y)u(y)dy}^2dx\cr
&\leq & C  \para{\int_{\R^n}\kappa_{A_1}(x)\para{\int_{\R^n}\kappa_{A_2}(y)\abs{\tilde{a}_t^{\rr}(x-y)}^2dy}dx}\norm{u}^2_{\LL}\cr
&\leq & \para{\int_{\R^n}\abs{\tilde{a}_t^{\rr}(z)}^2\para{\int_{\R^n}\kappa_{A_1}(x)\kappa_{A_2}(x-z)dx}dz}\norm{u}^2_{\LL}\cr
&\leq & \para{\int_{\R^n}\abs{\tilde{a}_t^{\rr}(z)}^2(\kappa_{A_1}*\check{\kappa}_{A_2})(z) dz}\norm{u}^2_{\LL},
\end{eqnarray}
where $\check{\kappa}_{A_2}(x)=\kappa_{A_2}(-x)$. \\
By a simple computation, we get
\begin{eqnarray}\label{2.25}
\tilde{a}_t^{\rr}(x)&=& e^{ix\cdot\rr} e^{-i\frac{t}{2}\abs{\rr}^2} \int_{\R^n}e^{i(x-t\rr)\cdot\xi} e^{-i\frac{t}{2}\abs{\xi}^2}a(\xi)d\xi\cr
&=& e^{ix\cdot\rr} e^{-i\frac{t}{2}\abs{\rr}^2}  \tilde{a}_t(x-t\rr).
\end{eqnarray}
Thus, we arrive at
\begin{eqnarray*}
\int_{\R^n}\abs{\tilde{a}_t^{\rr}(z)}^2(\kappa_{A_1}*\check{\kappa}_{A_2})(z) dz &\leq &\int_{\R^n}\abs{\tilde{a}_t(x)}^2(\kappa_{A_1}*\check{\kappa}_{A_2})(x+t\rr) dx\cr
&\leq & \int_{\R^n}\abs{\tilde{a}_t(x)}^2(\kappa_{A_1}*\check{\kappa}_{(A_2+t\rr)})(x) dx.
\end{eqnarray*}
Since, $\textrm{Supp}(\kappa_{A_1}*\check{\kappa}_{A_2+t\rr})\subset A_1-(A_2+t\rr)\subset\set{\abs{x}\geq \frac{1}{4}\abs{t\rr}}$, and
$$
\norm{\kappa_{A_1}*\check{\kappa}_{A_2}}_{L^\infty(\R^n)}\leq \norm{\kappa_{A_2}}_{L^1(\R^n)}\leq  C\abs{t\rr}^n,
$$
the above, inserted in \eqref{2.23} yields the following inequality
\begin{equation}\label{2.24}
\norm{\kappa_{A_1}e^{-itH_0}a_{-\rr}(D)\kappa_{A_2} u}^2_{\LL} \leq C\abs{t\rr}^n\para{\int_{\set{\abs{x}\geq \frac{1}{4}\abs{t\rr}  }}\abs{\tilde{a}_t (x)}^2 dx}\norm{u}^2_{\LL}.
\end{equation}
Let $r= \frac{1}{4}\abs{t\rr}$. In view of \eqref{2.25}, we get from Lemma \ref{L2.1} for $m\in\N$, with $2m-k>2n$, that
\begin{eqnarray}\label{2.26}
\abs{t\rr}^n\int_{\set{\abs{x}\geq \frac{1}{4}\abs{t\rr} }}\abs{\tilde{a}_t (x)}^2 dx 
&\leq & Cr^n\int_{\set{\abs{x}\geq r}} \abs{\tilde{a}_t (x)}^2 dx\cr
&\leq & C\seq{r}^{-k}\int_{\R^n}\seq{x}^{-2m+k+n}dx,
\end{eqnarray}
provided that $r>\abs{t}$, which satisfied if $\abs{\rr}>4$.\\
Combining \eqref{2.26} and \eqref{2.24}, we immediately deduce \eqref{2.20}.\\
This completes the proof.
\end{proof}
\begin{lemma}\label{L2.3}
Let $V\in\V$. Then for any $a\in\mathcal{C}_0^\infty(B(0,1))$ and every $\rr\in\R^n$, we have
\begin{equation}\label{2.27}
\norm{Ve^{-itH_0}a_{-\rr}(D)u}_{\LL}=\norm{V(x+t\rr)e^{-itH_0}a(D)e^{-ix\cdot\rr}u}_{\LL}
\end{equation}
for any $u\in\LL$.
\end{lemma}
\begin{proof}
By a density argument, it is enough to consider \eqref{2.27} for $u\in\mathcal{S}(\R^n)$. By \eqref{2.12}, we get
$$
\norm{Ve^{-itH_0}a_{-\rr}(D)u}_{\LL}=\norm{Ve^{-itH_0} e^{ix\cdot\rr}a(D)e^{-ix\cdot\rr }u}_{\LL}.
$$
Using \eqref{2.15} and \eqref{2.16}, we deduce that
\begin{eqnarray}\label{2.29}
\norm{Ve^{-itH_0} e^{ix\cdot\rr}a(D)e^{-ix\cdot\rr }u}_{\LL}&=&\norm{Ve^{ix\cdot\rr}\E e^{-itH_0} a(D)e^{-ix\cdot\rr }u}_{\LL}\cr
&=& \norm{E^t_{-\rr}V\E e^{-itH_0} a(D)e^{-ix\cdot\rr }u}_{\LL}\cr
&=&\norm{V(x+t\rr) e^{-itH_0} a(D)e^{-ix\cdot\rr }u}_{\LL}.
\end{eqnarray}
Thus we conclude the desired equality.
\end{proof}
\begin{lemma}\label{L2.4}
Let $V\in\V$. Then for any $a\in\mathcal{C}_0^\infty(B(0,1))$, and every $\rr\in\R^n$, $\abs{\rr}>4$, we have
\begin{equation}\label{2.31}
\norm{Ve^{-itH_0}a_{-\rr}(D)u}_{\LL}=\norm{V(x+t\rr)e^{-itH_0}a(D)e^{-ix\cdot\rr}u}_{\LL}\leq C\seq{t\rr}^{-\delta}\norm{u}_{L^2_\delta(\R^n)},
\end{equation}
for any $u\in L^2_\delta(\R^n)$.
\end{lemma}
\begin{proof}
By a density argument, it is enough to consider \eqref{2.31} for $u\in\mathcal{S}(\R^n)$. Let $A_1$ and $A_2$ are given by \eqref{2.19}. Then, we obtain
\begin{multline*}
\norm{Ve^{-itH_0}a_{-\rr}(D)u}_{\LL}\leq \norm{V\kappa_{A_1}e^{-itH_0}a_{-\rr}(D)u}_{\LL}\cr
+\norm{V\kappa_{A^c_1}e^{-itH_0}a_{-\rr}(D)u}_{\LL}:=I_1+I_2.
\end{multline*}
To estimate $I_1$, note that 
$$
I_1\leq \norm{V\kappa_{A_1}e^{-itH_0}a_{-\rr}(D)\kappa_{A_2}u}_{\LL}+\norm{V\kappa_{A_1}e^{-itH_0}a_{-\rr}(D)\kappa_{A^c_2}u}_{\LL}.
$$
Hence, by Lemma \ref{L2.2}, we get
\begin{eqnarray}\label{2.34}
\norm{V\kappa_{A_1}e^{-itH_0}a_{-\rr}(D)\kappa_{A_2}u}_{\LL} &\leq & \norm{\kappa_{A_1}e^{-itH_0}a_{-\rr}(D)\kappa_{A_2}u}_{\LL} \cr
&\leq & C\seq{t\rr}^{-\delta}\norm{u}_{\LL}.
\end{eqnarray}
Furthermore, for any $u\in\mathcal{S}(\R^n)$, one has
\begin{equation}\label{2.35}
\norm{V\kappa_{A_1}e^{-itH_0}a_{-\rr}(D)\kappa_{A^c_2}u}_{\LL}\leq \norm{\kappa_{A^c_2}u}_{\LL}\leq C\seq{t\rr}^{-\delta}\norm{u}_{L^2_\delta(\R^n)}.
\end{equation}
Taking into account \eqref{2.34}, \eqref{2.35}, we see that
\begin{equation}\label{2.36}
I_1\leq C\seq{t\rr}^{-\delta}\norm{u}_{L^2_\delta(\R^n)}.
\end{equation} 
On the other hand, since $A_1^c\subset\set{\abs{x}\geq \frac{1}{2}\abs{t\rr}}$ and $V\in\V$, we also have that
\begin{eqnarray}\label{2.37}
I_2=\norm{V\kappa_{A^c_1}e^{-itH_0}a_{-\rr}(D)u}_{\LL} &\leq &  C\seq{t\rr}^{-\delta}\norm{e^{-itH_0}a_{-\rr}(D)u}_{\LL}\cr
&\leq & C\seq{t\rr}^{-\delta}\norm{u}_{\LL}.
\end{eqnarray}
Hence, by combining \eqref{2.37} and \eqref{2.36}, we conclude the proof of the Lemma.
\end{proof}
\begin{lemma}\label{L2.5}
Assume that $V\in\V$. Then there exists $C>0$ such that for any $\Phi\in\mathcal{S}(\R^n)$ with $\textrm{Supp}(\hat{\Phi})\subset B(0,1)$, we have
$$
\norm{(W_\pm-I)e^{-itH_0}\Phi_\rr}_{\LL}\leq C\abs{\rr}^{-1}\norm{\Phi}_{L^2_\delta(\R^n)},
$$
for any $\rr\in\R^n$, $\abs{\rr}>4$, and uniformly for $t\in\R$. Here $\Phi_\rr=e^{ix\cdot\rr}\Phi$.
\end{lemma}
\begin{proof}
It follows from Duhamel's formula \eqref{2.7} that
$$
(W_+-I)e^{-itH_0}\Phi_\rr=i\int_0^\infty e^{isH}Ve^{-isH_0}e^{-itH_0}\Phi_\rr \,ds.
$$
Take $a\in\mathcal{C}_0^\infty(B(0,1))$, such that $a(\xi-\rr)\hat{\Phi}(\xi-\rr)=\hat{\Phi}(\xi-\rr)$, that is $a_{-\rr}(D)\Phi_\rr=\Phi_\rr$. Then by Lemma \ref{L2.4}, we get
\begin{eqnarray*}
\norm{(W_+-I)e^{-itH_0}\Phi_\rr }_{\LL} &\leq & \int_{-\infty}^{+\infty}\norm{Ve^{-is H_0}a_{-\rr}(D)\Phi_\rr  }_{L^2(\R^n)}ds\cr
&\leq & C\para{\int_\R\seq{s\rr}^{-\delta}ds}\norm{\Phi}_{L^2_\delta(\R^n)}\cr
&\leq & \frac{C}{\abs{\rr}}\para{\int_0^\infty\seq{\tau}^{-\delta}d\tau}\norm{\Phi}_{L^2_\delta(\R^n)},
\end{eqnarray*}
and the Lemma follows for $W_+$. The proof for $W_-$ is similar.
\end{proof}
\section{Stability of the X-ray transform of the potential}
\setcounter{equation}{0}
In this section we prove some estimate for the $X$-ray transform of the electric potential $V$. We start with a preliminary properties of the $X$-ray transform which needed to prove the main result.
\medskip

Let $\omega\in\s^{n-1}$, and $\omega^\bot$ the hyperplane through the origin orthogonal to $\omega$. We parametrize a line $\mathcal{L}(\omega,y)$ in $\R^n$ by specifying its direction $\omega\in \s^{n-1}$ and the point $y\in \omega^\bot$ where the line intersects the hyperpalne $\omega^\bot$. The $X$-ray transform of function $f\in L^1(\R^n)$ is given by
$$
X(f)(x,\omega)=X_\omega(f)(x)=\int_\R f(x+\tau\omega)d\tau,\quad x\in\omega^\bot.
$$
We see that $X(f)(x,\omega)$ is the integral of $f$ over the line $\mathcal{L}(\omega,y)$ parallel to $\omega$ which passes through $x\in\omega^\bot$. The following relation between the Fourier transform of $X_\omega(f)$ and $f$, called the Fourier slice theorem, will be useful: we denote by $\F_{\omega^\bot}$ the Fourier transform on function in the hyperplan $\omega^\bot$. The Fourier slice theorem is summarized in the following identity (see \cite{BJY}):  
\begin{eqnarray}\label{3.2}
\F_{\omega^\bot}(X_\omega(f))(\eta)&=&(2\pi)^{(1-n)/2}\int_{\omega^\bot} e^{-ix\cdot\eta}X_\omega(f)(x)dx\cr
&=&(2\pi)^{(1-n)/2}\int_{\omega^\bot} e^{-ix\cdot\eta}\int_\R f(x+s\omega)ds\,dx\cr
&=& (2\pi)^{(1-n)/2}\int_{\R^n} e^{-iy\cdot\eta} f(y)dy\cr
&=&\sqrt{2\pi}\F(f)(\eta),\quad \eta\in\omega^\bot.
\end{eqnarray}
The main purpose here is to present a preliminary estimate, which relates the difference of the short range potentials to the scattering map.
As before, we let $V_1,V_2\in\V$, $j=1,2$ be real valued potentials. We set
$$
V=V_1-V_2,
$$
such that
$$
\norm{V}_{\LL}\leq M.
$$
We start with the following Lemma.
\begin{lemma}\label{L3.1}
Let $V_j\in\V$, $j=1,2$. Then there exist $C>0$, $\lambda_0>0$ and $\gamma\in (0,1)$ such that for any $\omega\in\s^{n-1}$ and $\Phi,\Psi\in\mathcal{S}(\R^n)$ with $\textrm{Supp}(\hat{\Phi}),\,\textrm{Supp}(\hat{\Psi})\subset B(0,1)$, the following estimate holds true
\begin{multline}\label{3.3}
\abs{\int_{\R^n} X(V)(x,\omega)\Phi(x)\overline{\Psi}(x)dx}\leq C\sqrt{\lambda}\norm{S_{V_1}-S_{V_2}}\norm{\Phi}_{\LL}\norm{\Psi}_{\LL}\cr
+\lambda^{-\gamma/2}\para{\norm{\Phi}_{H^2(\R^n)}+\norm{\Phi}_{L^2_\delta(\R^n)}}\para{\norm{\Psi}_{H^2(\R^n)}+\norm{\Psi}_{L^2_\delta(\R^n)}}
\end{multline}
for any $\lambda>\lambda_0$. Here $V=V_1-V_2$.
\end{lemma}
\begin{proof}
Let $\rr=\sqrt{\lambda}\omega$ with $\lambda>0$ and $\omega\in\s^{n-1}$. In what follows for $\Phi,\,\Psi\in\mathcal{S}(\R^n)$ with $\textrm{Supp}(\hat{\Phi})$, and $\textrm{Supp}(\hat{\Psi})\subset B(0,1)$ we denote
$$
\Phi_\rr=e^{ix\cdot\rr}\Phi,\quad \Psi_\rr=e^{ix\cdot\rr}\Psi,\quad \rr=\sqrt{\lambda}\omega.
$$
From the identity \eqref{2.10} of the wave and scattering operators, it is easily seen that
\begin{equation}\label{3.5}
\sqrt{\lambda}\para{i(S_{V_j}-I)\Phi_\rr,\Psi_\rr}=\int_{\R}\ell_j(\tau,\lambda,\omega)d\tau+R_j(\lambda,\omega),\quad j=1,2,
\end{equation}
where the leading $\ell_j$, is given by
$$
\ell_j(\tau,\lambda,\omega):=\left(V_je^{-i\tau\lambda^{-1/2}H_0}\Phi_\rr, e^{-i\tau\lambda^{-1/2}H_0}\Psi_\rr\right),
$$
and the remainder term $R_j$, is given by
$$
R_j(\lambda,\omega)=\int_\R \left((W^j_--I)e^{-i\tau\lambda^{-1/2}H_0}\Phi_\rr, V_je^{-i\tau\lambda^{-1/2}H_0}\Psi_\rr\right)\,d\tau.
$$ 
At first, we estimate the remainder term $R_j$. Let $a\in\mathcal{C}_0^\infty(B(0,1))$ such that $a(\xi)\hat{\Phi}(\xi)=\hat{\Phi}(\xi)$, and $a(\xi)\hat{\Psi}(\xi)=\hat{\Psi}(\xi)$ Lemma \ref{L2.4} gives uniformly in $\lambda$ the integral bound
\begin{multline}\label{3.8}
\norm{V_je^{-i\tau\lambda^{-1/2}H_0}\Phi_\rr}=\norm{V_je^{-i\tau\lambda^{-1/2}H_0}a_{-\rr}(D)\Phi_\rr}\cr
=\norm{V_j(x+\tau\omega)e^{-i\tau\lambda^{-1/2}H_0}a(D)\Phi}\leq C\seq{\tau}^{-\delta}\norm{\Phi}_{L_\delta^2(\R^n)}.
\end{multline}
Similarly, we get
\begin{equation}\label{3.9}
\norm{V_je^{-i\tau\lambda^{-1/2}H_0}\Psi_\rr}\leq C\seq{\tau}^{-\delta}\norm{\Psi}_{L_\delta^2(\R^n)}.
\end{equation}
By Lemma \ref{L2.5} and \eqref{3.9}, we obtain
\begin{multline*}
\abs{R_j(\lambda,\omega)}\leq C\int_\R \norm{(W^j_--I)e^{-i\tau\lambda^{-1/2}H_0}\Phi_\rr}\norm{V_je^{-i\tau\lambda^{-1/2}H_0}\Psi_\rr}\,d\tau
\leq \frac{C}{\sqrt{\lambda}} \norm{\Phi}_{L_\delta^2(\R^n)}\norm{\Psi}_{L_\delta^2(\R^n)}.
\end{multline*}
Thus $R_j(\lambda,\omega)$ satisfies the remainder estimate in \eqref{3.3}.\\
We consider now the leading term $\ell_j$. Taking into account \eqref{2.15}, \eqref{2.16} a simple calculation gives
$$
\ell_j(\tau,\lambda,\omega)=\para{V_j(x+\tau\omega)e^{-i\tau\lambda^{-1/2}H_0}\Phi,e^{-i\tau\lambda^{-1/2}H_0}\Psi},
$$
and therefore
$$
\ell_j(\tau,\lambda,\omega)-\para{V_j(x+\tau\omega)\Phi,\Psi}=\ell^{(1)}_j(\tau,\lambda,\omega)+\ell^{(2)}_j(\tau,\lambda,\omega)
$$
where
$$
\ell^{(1)}_j(\tau,\lambda,\omega)= \para{V_j(x+\tau\omega)e^{-i\tau\lambda^{-1/2}H_0}\Phi,(e^{-i\tau\lambda^{-1/2}H_0}-I)\Psi},
$$
and 
$$
\ell^{(2)}_j(\tau,\lambda,\omega)= \para{(e^{-i\tau\lambda^{-1/2}H_0}-I)\Phi,V_j(x+\tau\omega)\Psi}.
$$
Since $\hat{\Psi}$ has compact support, we obtain
$$
\left(e^{-i\tau\lambda^{-1/2}H_0}-I\right)\Psi=\int_0^{\tau/\sqrt{\lambda}}\frac{d}{ds}(e^{-isH_0}\Psi)\,ds=-i\int_0^{\tau/\sqrt{\lambda}}e^{-isH_0}H_0\Psi\, ds.
$$
Therefore, we have
$$
\norm{ (e^{-i\tau\lambda^{-1/2}H_0}-I)\Psi}_{\LL}\leq \frac{\abs{\tau}}{\sqrt{\lambda}}\norm{H_0\Psi}_{\LL}\leq  \frac{\abs{\tau}}{\sqrt{\lambda}}\norm{\Psi}_{H^2(\R^n)},
$$
and using the fact that
$$
\norm{ (e^{-i\tau\lambda^{-1/2}H_0}-I)\Psi}_{\LL}\leq 2\norm{\Psi}_{\LL},
$$
we deduce the following estimation
\begin{equation}\label{3.17}
\norm{ (e^{-i\tau\lambda^{-1/2}H_0}-I)\Psi}_{\LL} \leq  C\para{\frac{\abs{\tau}}{\sqrt{\lambda}}}^{\gamma}\norm{\Psi}_{H^2(\R^n)},
\end{equation}
for all $\gamma\in (0,1)$. Then by \eqref{3.17} and \eqref{3.8}, we find
$$
\abs{\ell^{(1)}_j(\tau,\lambda,\omega)}\leq \frac{C}{\lambda^{\gamma/2}}\seq{\tau}^{-(\delta-\gamma)}\norm{\Psi}_{H^2(\R^n)}\norm{\Phi}_{L_\delta^2(\R^n)}.
$$
Hence, by selecting $\gamma$ small such that $\delta-\gamma>1$, it follow that 
\begin{equation}\label{3.19}
\int_\R\abs{\ell^{(1)}_j(\tau,\lambda,\omega)}\,d\tau\leq \frac{C}{\lambda^{\gamma/2}}\norm{\Psi}_{H^2(\R^n)}\norm{\Phi}_{L_\delta^2(\R^n)}.
\end{equation}
Moreover, we have
\begin{multline*}
\int_{\R^n}\abs{V_j(x+t\omega)\Psi(x)}^2dx\leq C\int_{\set{\abs{x+\tau\omega}>\frac{1}{2}\abs{\tau}}}\seq{x+\tau\omega}^{-2\delta}\abs{\Psi(x)}^2dx\cr
+\int_{\set{\abs{x+\tau\omega}\leq \frac{1}{2}\abs{\tau}}}\seq{x+\tau\omega}^{-2\delta}\abs{\Psi(x)}^2dx\cr
\leq C\para{\seq{\tau}^{-2\delta}\int_{\R^n}\abs{\Psi(x)}^2dx
+\seq{\tau}^{-2\delta}\int_{\R^n}\seq{x}^{2\delta}\abs{\Psi(x)}^2dx }\leq C\seq{\tau}^{-2\delta}\norm{\Psi}^2_{L_\delta^2(\R^n)}.
\end{multline*}
Then we show as the proof of \eqref{3.19} that
\begin{eqnarray*}
\int_\R\abs{\ell^{(2)}_j(\tau,\lambda,\omega)} d\tau &\leq & \int_\R \norm{(e^{-i\tau\lambda^{-1/2}H_0}-I)\Phi}_{\LL}\norm{V_j(x+\tau\omega)\Psi}_{\LL}d\tau\cr
&\leq & \frac{C}{\lambda^{\gamma/2}}\para{\int_\R\seq{\tau}^{-(\delta-\gamma)}d\tau}\norm{\Phi}_{H^2(\R^n)}\norm{\Psi}_{L_\delta^2(\R^n)}.
\end{eqnarray*}
Then, we easily see that
\begin{equation}\label{3.22}
\int_\R\abs{\ell^{(1)}_j(\tau,\lambda,\omega)} d\tau+\int_\R\abs{\ell^{(2)}_j(\tau,\lambda,\omega)} d\tau \leq \frac{C}{\lambda^{\gamma/2}}\para{\norm{\Psi}_{H^2(\R^n)}\norm{\Phi}_{L_\delta^2(\R^n)}+\norm{\Phi}_{H^2(\R^n)}\norm{\Psi}_{L_\delta^2(\R^n)} }.
\end{equation}
From \eqref{3.22} and \eqref{3.5}, we deduce that
$$
i\sqrt{\lambda}\para{(S_{V_1}-S_{V_2})\Phi_\rr,\Psi_\rr}=\int_\R(\ell_1-\ell_2)(\tau,\lambda,\omega)d\tau+(R_1-R_2)(\lambda,\omega):=\int_\R\ell(\tau,\lambda,\omega)d\tau+R(\lambda,\omega) ,
$$
where the leading and remainder terms, respectively, satisfy
\begin{multline*}
\abs{\int_\R\para{\ell(\tau,\lambda,\omega)-(V(x+\tau\omega)\Phi,\Psi)} d\tau}\leq \sum_{j=1}^2\int_\R\para{\abs{\ell^{(1)}_j(\tau,\lambda,\omega)}+\abs{\ell^{(2)}_j(\tau,\lambda,\omega)} } d\tau\cr
\leq \frac{C}{\lambda^{\gamma/2}}\para{\norm{\Psi}_{H^2(\R^n)}\norm{\Phi}_{L_\delta^2(\R^n)}+\norm{\Phi}_{H^2(\R^n)}\norm{\Psi}_{L_\delta^2(\R^n)} },
\end{multline*}
and
$$
\abs{R(\lambda,\omega)}\leq \abs{R_1(\lambda,\omega)}+\abs{R_2(\lambda,\omega)}\leq \frac{C}{\sqrt{\lambda}}\norm{\Phi}_{L_\delta^2(\R^n)}\norm{\Psi}_{L_\delta^2(\R^n)}.
$$
This completes the proof of Lemma \ref{L3.1}.
\end{proof}
\section{Proof of the stability estimate}
\setcounter{equation}{0}
In this section, we complete the proof of Theorem \ref{T.1}. We are going to use the estimate proved in the previous section; this will provide information on the $X$-ray transform of the difference of short range electric potentials.
\medskip

Let $\omega\in \s^{n-1}$ and $V\in\V$. We denote
$$
f(x)=X(V)(x,\omega)=\int_\R V(x+t\omega)dt.
$$
Then $f$ satisfies the following estimate
\begin{eqnarray*}
\abs{f(x)}=\abs{f(x-(\omega\cdot x)\omega)} &\leq&  C \int_\R\seq{x-(\omega\cdot x)\omega+t\omega}^{-\delta} dt \cr
&\leq & \frac{C}{\seq{x-(x\cdot\omega)\omega}^{\delta-1}}\int_\R\seq{t}^{-\delta}dt,\quad \forall x\in\R^n.
\end{eqnarray*}
In particularly, we have $f\in L^1(\omega^\bot)$.
\smallskip

For any $\Phi$, $\Psi$ with $\textrm{Supp}(\hat{\Phi})\subset B(0,1)$ and  $\textrm{Supp}(\hat{\Psi})\subset B(0,1)$, we have by \eqref{3.3}
\begin{multline*}
\abs{\int_{\R^n} f(x)\Phi(x)\overline{\Psi}(x)dx}\leq \sqrt{\lambda}\norm{S_{V_1}-S_{V_2}}\norm{\Phi}_{\LL}\norm{\Psi}_{\LL}\cr
+C\lambda^{-\gamma/2}\para{\norm{\Phi}_{H^2(\R^n)}+\norm{\Phi}_{L^2_\delta(\R^n)}}\para{\norm{\Psi}_{H^2(\R^n)}+\norm{\Psi}_{L^2_\delta(\R^n)}}.
\end{multline*}
Let $\eta\in\omega^\bot$ be fixed and let $\Psi\in \LL$ such that  $\textrm{Supp}(\hat{\Psi})\subset B(\eta/2,1)$. Denote by $\Psi_{\eta/2}=e^{ix\cdot\eta/2}\Psi$, then  $\textrm{Supp}(\hat{\Psi}_{\eta/2})\subset B(0,1)$. Applying the last inequality with $\Psi=\Psi_{\eta/2}$, we find
\begin{multline}\label{4.4}
\abs{\int_{\R^n} f_{-\eta/2}(x)\Phi(x)\overline{\Psi}(x)dx}\leq \sqrt{\lambda}\norm{S_{V_1}-S_{V_2}}\norm{\Phi}_{\LL}\norm{\Psi}_{\LL}\cr
+C\seq{\eta}^2\lambda^{-\gamma/2}\para{\norm{\Phi}_{H^2(\R^n)}+\norm{\Phi}_{L^2_\delta(\R^n)}}\para{\norm{\Psi}_{H^2(\R^n)}+\norm{\Psi}_{L^2_\delta(\R^n)}},
\end{multline}
where $f_{-\eta/2}=e^{-ix\cdot\eta/2}f$.
\begin{lemma}\label{L4.1}
Let $V_j\in\V$, $j=1,2$. Then there exist $C>0$, $\lambda_0>0$, $\gamma\in (0,1)$ and $\sigma>n/2+\gamma$ and  such that for any $\omega\in\s^{n-1}$ and $\Phi\in\mathcal{S}(\R^n)$ such that $\textrm{Supp}(\hat{\Phi})\subset B(0,1)$, the following estimate holds true
\begin{multline*}
\abs{\F(f\Phi)(\eta)}\leq \varepsilon^{-n/2}\sqrt{\lambda}\norm{S_{V_1}-S_{V_2}}\norm{\Phi}_{\LL}\cr
+C\seq{\eta}^4\lambda^{-\gamma/2}\varepsilon^{-n/2-\delta}\para{\norm{\Phi}_{H^2(\R^n)}+\norm{\Phi}_{L^2_\delta(\R^n)}}+C\varepsilon^\gamma\norm{\Phi}_{L^2_\sigma(\R^n)}
\end{multline*}
for any $\lambda>\lambda_0$, $\eta\in\omega^\bot$ and $\varepsilon\in (0,1)$.
\end{lemma}
\begin{proof}
Let $\psi_0\in\mathcal{C}_0^\infty(B(0,1))$, with $\norm{\psi_0}_{L^1(\R^n)}=1$, we define
$$
\psi_\varepsilon(\xi)=\varepsilon^{-n}\psi_0(\varepsilon^{-1}(\xi-\eta/2)),\quad \textrm{Supp}(\psi_\varepsilon)\subset B(\eta/2,\varepsilon)\subset B(\eta/2,1),
$$
and let $\Psi_\varepsilon=\F^{-1}(\psi_\varepsilon)$. By Plancherel formula, we get
\begin{equation}\label{4.6}
\int_{\R^n} f_{-\eta/2}(x)\Phi(x)\overline{\Psi}_\varepsilon (x)dx=\int_{\R^n}\F(f_{-\eta/2}\Phi)(\xi)\psi_\varepsilon(\xi)d\xi.
\end{equation}
Taking into account \eqref{4.6} and applying \eqref{4.4} with $\Psi=\Psi_\varepsilon$, we obtain
\begin{multline}\label{4.7}
\abs{\int_{\R^n} \F(f_{-\eta/2}\Phi)(\xi)\psi_\varepsilon(\xi)d\xi}\leq \sqrt{\lambda}\norm{S_{V_1}-S_{V_2}}\norm{\Phi}_{\LL}\norm{\psi_\varepsilon}_{\LL}\cr
+C\seq{\eta}^2\lambda^{-\gamma/2}\para{\norm{\Phi}_{H^2(\R^n)}+\norm{\Phi}_{L^2_\delta(\R^n)}}\para{\norm{\Psi_\varepsilon}_{H^2(\R^n)}+\norm{\Psi_\varepsilon}_{L^2_\delta(\R^n)}}.
\end{multline}
Furthermore there exists $C>0$ such that 
\begin{equation}\label{4.8}
\norm{\psi_\varepsilon}_{\LL}^2=\varepsilon^{-n}\int_{\R^n}\abs{\psi_0(\xi)}^2 d\xi\leq C\varepsilon^{-n},
\end{equation}
and
\begin{equation}\label{4.9}
\norm{\Psi_\varepsilon}^2_{H^2(\R^n)}=\int_{\R^n}\seq{\xi}^4\abs{\psi_\varepsilon(\xi)}^2d\xi=\varepsilon^{-2n}\int_{\R^n}\seq{\xi}^4\abs{\psi_0(\varepsilon^{-1}(\xi-\eta/2))}^2d\xi\leq C\varepsilon^{-n}\seq{\eta}^4.
\end{equation}
Using the fact that
$$
\F(\psi_\varepsilon)(y)=e^{-iy\cdot\eta/2}\hat{\psi}_0(\varepsilon y)
$$
and \eqref{2.3.1}, we get
\begin{equation}\label{4.10}
\norm{\Psi_\varepsilon}_{L^2_\delta(\R^n)}=\norm{\psi_\varepsilon}_{H^\delta(\R^n)}\leq C\varepsilon^{-n/2-\delta}.
\end{equation} 
Then, by \eqref{4.7}, \eqref{4.8}, \eqref{4.9} and \eqref{4.10}, one gets
\begin{multline*}
\abs{\int_{\R^n} \F(f_{-\eta/2}\Phi)(\xi)\psi_\varepsilon(\xi)d\xi}\leq \sqrt{\lambda}\varepsilon^{-n/2}\norm{S_{V_1}-S_{V_2}}\norm{\Phi}_{\LL}\cr
+C\seq{\eta}^4\lambda^{-\gamma/2}\varepsilon^{-n/2-\delta}\para{\norm{\Phi}_{H^2(\R^n)}+\norm{\Phi}_{L^2_\delta(\R^n)}}.
\end{multline*}
Moreover, we have
\begin{multline*}
\F(f_{-\eta/2}\Phi)(\eta/2)=\int_{\R^n}\F(f_{-\eta/2}\Phi)(\xi)\psi_\varepsilon(\xi)d\xi\cr
+\int_{\R^n}\para{\F(f_{-\eta/2}\Phi)(\eta/2)-\F(f_{-\eta/2}\Phi)(\xi) }\psi_\varepsilon(\xi)d\xi.
\end{multline*}
Furthermore, for any $\gamma\in (0,1)$, there exists $C=C(\gamma)>0$ such that
\begin{eqnarray*}
\abs{\F(f_{-\eta/2}\Phi)(\eta/2)-\F(f_{-\eta/2}\Phi)(\xi) } &\leq & C\abs{\xi-\eta/2}^\gamma\int_{\R^n}\seq{x}^\gamma\abs{\Phi(x)} dx\cr
&\leq & C\abs{\xi-\eta/2}^\gamma\para{\int_{\R^n}\seq{x}^{-2\sigma+2\gamma}dx}^{1/2}\norm{\Phi}_{L^2_\sigma(\R^n)},
\end{eqnarray*}
for some $\sigma>\gamma+n/2$. We deduce that

\begin{multline*}
\abs{\int_{\R^n}\para{\F(f_{-\eta/2}\Phi)(\eta/2)-\F(f_{-\eta/2}\Phi)(\xi) }\psi_\varepsilon(\xi)d\xi}\cr
\leq C \norm{\Phi}_{L^2_\sigma(\R^n)}\int_{\R^n}\abs{\eta/2-\xi}^\gamma\abs{\psi_\varepsilon(\xi)}d\xi\leq C\varepsilon^\gamma \norm{\Phi}_{L^2_\sigma(\R^n)},
\end{multline*}
which imply 
\begin{multline*}
\abs{\F(f\Phi)(\eta)}=\abs{ \F(f_{-\eta/2}\Phi)(\eta/2)}\leq \varepsilon^{-n/2} \sqrt{\lambda}\norm{S_{V_1}-S_{V_2}}\norm{\Phi}_{\LL}\cr
+\seq{\eta}^4\lambda^{-\gamma/2}\varepsilon^{-n/2-\delta}\para{\norm{\Phi}_{H^2(\R^n)}+\norm{\Phi}_{L^2_\delta(\R^n)}}+C\varepsilon^\gamma\norm{\Phi}_{L^2_\sigma (\R^n)}.
\end{multline*}
This completes the proof of the Lemma.
\end{proof}
We give now the following Lemma to be used later
\begin{lemma}\label{L4.2}
Let $\theta\in\mathcal{C}_0^\infty((-\frac{1}{2},\frac{1}{2}))$ and $\varphi\in\mathcal{C}_0^\infty(\omega^\bot\cap B(0,\frac{1}{2}))$. Putting
$$
\Phi(y)=\F_0^{-1}(\theta)(y\cdot\omega)\F_{\omega^\bot}^{-1}(\varphi)(y-(y\cdot\omega)\omega),\quad y\in\R^n,
$$
where $\F_0$ denote the Fourier transform on function in $\R$. Then we have $\textrm{Supp}(\hat{\Phi})\subset B(0,1)$ and 
$$
\hat{\Phi}(\xi)=\theta(\omega\cdot\xi)\varphi(\xi-(\omega\cdot\xi)\omega),\quad \forall \xi\in\R^n.
$$
Moreover,  for all $s\geq 0$, we have 
$$
\norm{\Phi}_{H^s(\R^n)}\leq \norm{\theta}_{L^2_s(\R)}\norm{\varphi}_{L^2_s(\omega^\bot)}.
$$
Finally, for any $\delta\geq 0$, there exists $C>0$ such that 
$$
\norm{\Phi}_{L^2_\delta(\R^n)}\leq C\norm{\varphi}_{H^\delta(\omega^\bot)}.
$$
 Here $C$ depends on norms of $\theta$.
\end{lemma}
The next step in the proof is to deduce an estimate that links the Fourier transform of the unknown coefficient to the measurement $S_{V_1}-S_{V_2}$.

\begin{lemma}\label{L4.3}
Let $V_j\in\V$, $j=1,2$. Then there exist $C>0$, $\lambda_0>0$, $\gamma\in (0,1)$ and $\alpha_j>0$, $j=1,2,3$, such that for any $\omega\in\s^{n-1}$ the following estimate holds true
\begin{equation}\label{4.15}
\abs{\F_{\omega^\bot}(f)(\eta)}\leq C\varepsilon^{-\alpha_1}\sqrt{\lambda}\norm{S_{V_1}-S_{V_2}}+\lambda^{-\gamma/2}\varepsilon^{-\alpha_2}\seq{\eta}^4+C\varepsilon^{\alpha_3},
\end{equation}
for any $\lambda>\lambda_0$, $\eta\in\omega^\bot$ and $\varepsilon\in (0,1)$.
\end{lemma}
\begin{proof}
Let $\theta\in\mathcal{C}_0^\infty(-1/4,1/4)$ and $\varphi\in\mathcal{C}_0^\infty(\omega^\bot\cap B(0,1/2))$. Putting
$$
\Phi(y)=\F_0^{-1}(\theta)(y\cdot\omega)\F^{-1}_{\omega^\bot}(\varphi)(y-(y\cdot\omega)\omega),\quad y\in\R^n.
$$
We assume further $\theta(0)=1$. Then we have by Lemma \ref{L4.2} $\textrm{Supp}(\hat{\Phi})\subset B(0,1)$. The change of variable $x=y+t\omega\in\omega^\bot\oplus\R\omega$, $dx=dydt$ yields, after noting that $\eta\in\omega^\bot$
\begin{eqnarray}\label{4.16}
\F(f\Phi)(\eta)&=&(2\pi)^{-n/2}\int_{\R^n}e^{-ix\cdot\eta}\F_0^{-1}\theta(x\cdot\omega)\F^{-1}_{\omega^\bot}(\varphi)(x-(x\cdot\omega)\omega)f(x)dx\cr
&=&(2\pi)^{-n/2}\int_{\R}\int_{\omega^\bot}e^{-iy\cdot\eta}\F^{-1}_0(\theta)(t)\F^{-1}_{\omega^\bot}(\varphi)(y)f(y)dy\,dt\cr
&=&(2\pi)^{-(n-1)/2}\int_{\omega^\bot}e^{-iy\cdot\eta}\F^{-1}_{\omega^\bot}(\varphi)(y)f(y)dy\cr
&=&\F_{\omega^\bot}(f\F_{\omega^\bot}^{-1}(\varphi))(\eta)=\F_{\omega^\bot}(f)*\varphi(\eta),
\end{eqnarray}
where we have used $f(y-t\omega)=f(y)$ for any $t\in\R$. Taking account \eqref{4.16} and applying Lemma \ref{L4.2}, one gets
\begin{multline}\label{4.19}
\abs{\int_{\omega^\bot}\F_{\omega^\bot}(f)(\xi)\varphi(\eta-\xi)d\xi}\leq \varepsilon^{-n/2}\sqrt{\lambda}\norm{S_{V_1}-S_{V_2}}\norm{\varphi}_{L^2(\omega^\bot)}\cr
+C\seq{\eta}^4\lambda^{-\gamma/2}\varepsilon^{-n/2-\delta}\para{\norm{\varphi}_{L_2^2(\omega^\bot)}+\norm{\varphi}_{H^\delta(\omega^\bot)}}+C\varepsilon^\gamma\norm{\varphi}_{H^\sigma(\omega^\bot)}.
\end{multline}
Now, we specify the choice of the function $\varphi$. Let $\varphi_0\in\mathcal{C}_0^\infty(\omega^\bot\cap B(0,1/2))$ with 
$\norm{\varphi_0}_{L^1(\omega^\bot)}=1$, we define, for $h$ small
$$
\varphi_h(\xi)=h^{-n+1}\varphi_0(h^{-1}\xi),\quad \xi\in\omega^\bot.
$$
Applying \eqref{4.19} with $\varphi=\varphi_h$, we get
\begin{multline*}
\abs{\int_{\omega^\bot}\F_{\omega^\bot}(f)(\xi)\varphi_h(\eta-\xi)d\xi}\leq \varepsilon^{-n/2}\sqrt{\lambda}\norm{S_{V_1}-S_{V_2}}\norm{\varphi_h}_{L^2(\omega^\bot)}\cr
+C\seq{\eta}^4\lambda^{-\gamma/2}\varepsilon^{-n/2-\delta}\para{\norm{\varphi_h}_{L_2^2(\omega^\bot)}+\norm{\varphi_h}_{H^\delta(\omega^\bot)}}+C\varepsilon^\gamma\norm{\varphi_h}_{H^\sigma(\omega^\bot)}.
\end{multline*}
Since
$$
\norm{\varphi_h}_{L^2(\omega^\bot)}=h^{(1-n)/2}\norm{\varphi_0}_{L^2(\omega^\bot)},\quad \norm{\varphi_h}_{L_2^2(\omega^\bot)}\leq C h^{(1-n)/2}\norm{\varphi_0}_{L_2^2(\omega^\bot)}
$$
and, 
$$
\norm{\varphi_h}_{H^\sigma(\omega^\bot)}\leq C h^{-\sigma+(1-n)/2}\norm{\varphi_0}_{H^\sigma(\omega^\bot)},
$$
we obtain
\begin{multline*}
\abs{\int_{\omega^\bot}\F_{\omega^\bot}(f)(\xi)\varphi_h(\eta-\xi)d\xi}\leq \varepsilon^{-n/2}\sqrt{\lambda}h^{(1-n)/2}\norm{S_{V_1}-S_{V_2}}\cr
+C\seq{\eta}^4\lambda^{-\gamma/2}\varepsilon^{-n/2-\delta} h^{-\delta+(1-n)/2}+C\varepsilon^\gamma h^{-\sigma+(1-n)/2}.
\end{multline*}
Moreover
$$
\F_{\omega^\bot}(f)(\eta)=\int_{\omega^\bot}\F_{\omega^\bot}(f)(\xi)\varphi_h(\eta-\xi)d\xi-\int_{\omega^\bot}(\F_{\omega^\bot}(f)(\xi)-\F_{\omega^\bot}(f)(\eta))\varphi_h(\eta-\xi)d\xi.
$$
Using the fact that,
\begin{eqnarray*}
\abs{\F_{\omega^\bot}(f)(\xi)-\F_{\omega^\bot}(f)(\eta)} &\leq & C\int_{\omega^\bot} \abs{e^{-ix\cdot\xi}-e^{-ix\cdot\eta}}\abs{f(x)}dx\cr
&\leq & C\abs{\xi-\eta}^{\gamma'}\int_{\omega^\bot}\seq{x}^{\gamma'}\abs{f(x)}dx\cr
&\leq & C\abs{\xi-\eta}^{\gamma'}\norm{V}_{L^1_{\gamma'}(\R^n)},
\end{eqnarray*}
with $\gamma'>0$ sufficiently small. We deduce that
\begin{eqnarray*}
\abs{\int_{\omega^\bot}(\F_{\omega^\bot}(f)(\xi)-\F_{\omega^\bot}(f)(\eta))\varphi_h(\eta-\xi)d\xi} &\leq & CM\int_{\omega^\bot}\abs{\xi-\eta}^{\gamma'}\abs{\varphi_h(\eta-\xi)}d\xi\cr
&\leq & CMh^{\gamma'}.
\end{eqnarray*}
We obtain, for any $\eta\in\omega^\bot$
\begin{eqnarray*}
\abs{\F_{\omega^\bot}(f)(\eta)} 
&\leq & \varepsilon^{-n/2}\sqrt{\lambda}h^{(1-n)/2}\norm{S_{V_1}-S_{V_2}}+\lambda^{-\gamma/2}\varepsilon^{-n/2-\delta}\seq{\eta}^4 h^{-\delta+(1-n)/2}\cr
&&+C\varepsilon^\gamma h^{-\sigma+(1-n)/2}+Ch^{\gamma'}.
\end{eqnarray*}
Selecting $h$ such that $\varepsilon^\gamma h^{-\sigma+(1-n)/2}=h^{\gamma'}$, we obtain
$$
\abs{\F_{\omega^\bot}(f)(\eta)} \leq \varepsilon^{-\alpha_1}\sqrt{\lambda}\norm{S_{V_1}-S_{V_2}}+\lambda^{-\gamma/2}\varepsilon^{-\alpha_2}\seq{\eta}^4+C\varepsilon^{\alpha_3}.
$$
This completes the proof of the Lemma.
\end{proof}
We return now to the proof of Theorem \ref{T.1}. Since $\omega$ is arbitrary, we deduce from \eqref{4.15} and \eqref{3.2}
\begin{equation}\label{4.24}
\abs{\F(V)(\eta)}\leq  \varepsilon^{-\alpha_1}\sqrt{\lambda}\norm{S_{V_1}-S_{V_2}}+\lambda^{-\gamma/2}\varepsilon^{-\alpha_2}\seq{\eta}^4 +C\varepsilon^{\alpha_3},\quad \forall\eta\in\R^n.
\end{equation}
In light of the above reasoning and decomposing the $H^{-1}(\R^n)$ norm of $V$ as
$$
\norm{V}^2_{H^{-1}(\R^n)}=\int_{\abs{\eta}\leq R} \seq{\eta}^{-2}\abs{\F(V)(\eta)}^2d\eta+ \int_{\abs{\eta}> R} \seq{\eta}^{-2}\abs{\F(V)(\eta)}^2d\eta
$$
then, by \eqref{4.24}, we get
$$
\norm{V}^2_{H^{-1}(\R^n)}\leq C\para{R^n(\varepsilon^{-\alpha_1}\sqrt{\lambda}\norm{S_{V_1}-S_{V_2}}+\lambda^{-\gamma/2}\varepsilon^{-\alpha_2}R^2 +C\varepsilon^{\alpha_3})+\frac{M^2}{R^2} }.
$$
The next step is to choose in such away $\varepsilon^{\alpha_3} R^n=R^{-2}$. In this case we get
$$
\norm{V}^2_{H^{-1}(\R^n)}\leq C\para{R^{\beta_1}\sqrt{\lambda}\norm{S_{V_1}-S_{V_2}}+\lambda^{-\gamma/2}R^{\beta_2}+\frac{1}{R^2} }.
$$
Now we choose $R>0$ in such that away $\lambda^{-\gamma/2}R^{\beta_2}=R^{-2}$. In this case we get
$$
\norm{V}^2_{H^{-1}(\R^n)}\leq C\para{\lambda^{\mu_1}\norm{S_{V_1}-S_{V_2}}+\lambda^{-\mu_2}},
$$
for some positive constants $\mu_1,\mu_2$. Finally, minimizing the right hand side with respect to $\lambda$ we obtain the desired estimate of Theorem \ref{T.1}.

\end{document}